\documentclass[12pt]{amsart}
\usepackage{latexsym,fancyhdr,amssymb,color,amsmath,amsthm,graphicx,listings,comment}
\usepackage[section]{placeins}
\newtheorem{thm}{Theorem}
\newtheorem{lemma}{Lemma}
\newtheorem{coro}{Corollary}
\let\paragraph\subsection
\definecolor{yellow1}{rgb}{1,1,0.8}

\title{The Sphere Formula}
\author{Oliver Knill}
\date{January 12, 2023}
\address{Department of Mathematics \\ Harvard University \\ Cambridge, MA, 02138 }


\begin{document}
\maketitle

\begin{abstract}
The sphere formula states that in an arbitrary finite abstract simplicial complex, the sum of the Euler 
characteristic of unit spheres centered at even-dimensional simplices is equal to the sum of the Euler 
characteristic of unit spheres centered at odd-dimensional simplices. It follows that if a geometry has
constant unit sphere Euler characteristic, like a manifold, then either all its unit spheres have zero Euler characteristic
or the space itself has zero Euler characteristic. Especially, odd-dimensional manifolds have zero Euler characteristic,
a fact usually verified either in algebraic topology using  Poincar\'e duality together with Riemann-Hurwitz 
then deriving it from the existence of a Morse function, using that the Morse indices of the function
and its negative add up to zero in odd dimensions.
Gauss Bonnet also shows that odd-dimensional Dehn-Sommerville spaces have zero
Euler characteristic because they have constant zero curvature. 
Zero curvature phenomenons can be understood integral geometrically as index expectation
or as Dehn-Sommerville relations.
\end{abstract}

\section{Energy Formula} 

\paragraph{}
A finite abstract simplicial complex $G$ is a finite set of non-empty sets 
that is closed under the operation of taking finite non-empty subsets. It has a 
{\bf finite topology} in which the set of {\bf stars} $U(x)=\{y, x \subset y\}$ 
forms a topological basis. 
The {\bf cores} $K(x) = \{ y, y \subset x\}$ are closed sets in this topology. 
Similarly as the Zariski topology, it is not-Hausdorff. Indeed, 
incident elements in $G$ can not be separated by open sets. 
But the topology has the desired connectivity properties.
The closure $B(x)$ of a star $U(x)$ is called the {\bf unit ball}.
The topological boundary $S(x)=B(x) \setminus U(x)$ of $B(x)$ is the unit {\bf sphere}
and sometimes called the {\bf link} of $x$. 
Define $w(x)=(-1)^{{\rm dim}(x)}$ and for $A \subset G$, let 
$w(G)=\sum_{x \in A} w(x)$ denote the {\bf Euler characteristic} of $A$. 

\paragraph{}
Let us start with an energy formula which expresses
Euler characteristic as the total potential energy $\sum_{x,y} g(x,y)$ between any pair of simplices.
The potential energy $g(x,y)$ is the Green function
$$   g(x,y) = w(x) w(y) w(U(x) \cap U(y)) \;.  $$ 
This defines a unimodular $n \times n$ matrix if $G$ has $n$ elements 
\cite{Unimodularity,MukherjeeBera2018,CountingMatrix,GreenFunctionsEnergized,EnergizedSimplicialComplexes}.
This matrix is the inverse of the connection Laplacian 
$L(x,y) = 1$ if $x \cap y \neq \emptyset$ and $L(x,y)=0$ else. 
If $\sum_{x,y} g(x,y)$ is the total energy, we had shown that the potential at a point
$V(x)=\sum_y g(x,y)$ can be simplified to $w(x) w(U(x))$. The Euler characteristic of 
the smallest open set containing $x$ is so linked to the potential at $x$ which also can 
be seen as a {\bf curvature}, dual to $w(x)$ at $x$. Let us just directly and independently verify.

\begin{thm}[Energy formula]
$\sum_{x \in G} w(x) w(U(x)) = w(G)$. 
\end{thm}
\begin{proof}
Replace the $\pm 1$-valued function $w(x)$ by a more general function $h(x)$ and define
$w_h(G) = \sum_{x} h(x)$ and $g_h(x) = w(x) w_h(U(x))$.
The map $h \to g_h$ is linear and the energy formula is linear so that the formula only needs to 
be shown in the case when $h(x_0)=1$ for $x \neq x_0$. The right hand side
is then $w_h(G) =1$. Now look at the left hand side $\sum_{x \in G} w(x) w(U(x))$. This is a sum
over all $w(x)$ for which $x_0 \in U(x)$ which means $\sum_{x, x \subset x_0} w(x)$
but this is the Euler characteristic of the simplicial complex 
$\{ x \subset x_0 \} = \overline{\{ x_0 \}}$ which is always $1$. So, also the left hand side is $1$. 
\end{proof} 

\section{Sphere Formula}

\paragraph{}
Euler characteristic is a very special particular functional on simplicial complexes. It is a point
in a $(n+1)$-dimensional space of {\bf valuations} $F$, functions that satisfy 
$F(A) + F(B) = F(A \cup B) + F(A \cap B)$ for any subsets of $A$ and $B$.

\begin{lemma}[Local valuation] 
$w(B(x))=w(U(x))+w(S(x))$
\end{lemma}

\begin{proof} 
For any subsets $A,B \subset G$, whether open or closed or neither, we have
the {\bf valuation formula} $w(A) + w(B) = w(A \cup B) - w(A \cap B)$ because
each of the basic valuations $f_k(G)$ counting the number of $k$-dimensional 
simplices satisfies the formula and $w$ is a linear combination of such 
basic valuations. 
\end{proof} 

\paragraph{}
The {\bf unit ball} $B(x)$, defined as the closure of the star $U(x) = \{ y, x \subset y\}$
is a closed set which contains $x$ as well as its simplicial complex, the core $\overline{\{x\}} 
= \{ y, y \subset x\}$. 

\begin{lemma}[Unit balls]
$w(B(x))=1$ for all $x$.
\end{lemma}

\begin{proof}
Use induction with respect to the number of elements in $B(x)$.
If $B(x)$ has one element, it has $w(B(x))=1$. We can reduce the size of $B(x)$ by 
taking a way an element $y \in B(x)$ different from $x$. The complex $B'(x)=B(x) \setminus U(y)$ is 
now a unit ball $B'(x)$ in a smaller $G \setminus U(y)$. The induction assumption assures
that $w(B'(x))=1$. The local valuation formula gives that $w(B(x)))=w(B'(x)$. 
\end{proof} 

\paragraph{}
We get for all complexes (see \cite{KnillEnergy2020} Corollary 6). 

\begin{thm}[Sphere Formula]  
$\sum_{x \in G} w(x) w(S(x)) = 0$.
\end{thm}
\begin{proof} 
The property $\sum_{x \in G} w(x) w(B(x)) = \sum_{x \in G} w(x) = w(G)$ 
reduces so to the definition of $w$. 
Subtract $\sum_{x \in G} w(x) w(U(x))=w(G)$ from 
the sphere formula $\sum_{x \in G} w(x) w(B(x))=w(G)$ and use 
the local {\bf valuation formula}: $w(B(x))=w(U(x))+w(S(x))$. 
\end{proof} 

\paragraph{}
In analogy to the Green function $g(x,y) = w(x) w(y) w((U(x) \cap U(y))$ whose super trace
is $w(G)$, we could look at the sphere Green matrix $s(x,y) = w(x) w(y) w(S(x) \cap S(y))$. The 
sphere theorem tells that the super trace of $s$ is $0$. In our experiments, we notice
the unexplained fact that the matrix $s$ is always singular. The reader is invited to 
experiment with the code provided below. This is in contrast to the 
Green matrix $g$ which is always {\bf unimodular}, meaning that the determinant is $1$ or $-1$.  

\paragraph{}
Let us mention in this context a discrete analogue of a {\bf theorem of Hadwiger} stating that the 
linear space of valuations has dimension ${\rm dim}(G)+1$ and that the $f$-vectors 
$f_k(G)$ counting the $k$-dimensional simplices form a basis for $k=0, \dots,{\rm dim}(G)$. 
Euler characteristic $w(G)=\sum_{k=0}^{{\rm dim}(G)} (-1)^k f_k(G)$ is the only valuation
which preserves the Barycentric
refinement operation, an operation which induces a linear map $f \to Qf$ on $f$ vectors and
so a map $w \to Q^T w$ with a concrete matrix $Q(x,y) = {\rm Stirling}(x,y) x!$ which has a unique 
eigenvalue $1$. The corresponding eigenvector for $Q^T$ is $(1,-1,1,-1,...)$. 
The dynamics of Brycentric refinement is interesting from a spectral point of view
\cite{KnillBarycentric,KnillBarycentric2}.

\section{Euler's Gem}

\paragraph{}
For $x \in G$, the {\bf index} $i(x)=1-w(S(x))$ tells how the Euler characteristic
changes if the open set $U(x)$ is taken away from $G$. If $H=G \setminus U(x)$, then 
$$ w(G) = w(H)+w(B(x))-w(S(x)) = w(H)+w(U(x)) = w(H) + i(x)  \;.  $$
This means that under reduction $G \to H = G \setminus U(x)$, the Euler characteristic
changes as $w(H) = w(G)-i(x)$. This process leads rather general Poincar\'e-Hopf theorems.
We will mention below a version for simplicial complexes. 

\paragraph{}
Inductively, a complex $G$ is called {\bf contractible} if there exists $x \in G$
such that the both closed sets $S(x)$ as well as $G \setminus U(x)$
are both contractible. To found the induction, the {\bf one point complex} $G=1$ 
is declared to be contractible. We have just seen in the proof of the unit ball 
lemma that every unit ball is contractible.

\paragraph{}
A complex $G$ is called a {\bf $d$-manifold}, if every unit sphere $S(x)$ is a {\bf $(d-1)$ sphere}.
A complex $G$ is called a {\bf $d$-sphere}, if it is a $d$-manifold and $G \setminus U(x)$
is contractible for some $x$. The empty complex is $G=0$ is declared to be the {\bf $(-1)$ sphere}.

\begin{thm}[Euler Gem] 
$w(G) =1+(-1)^d$ for every $d$-sphere $G$.
\end{thm}

\begin{proof}
By definition of a sphere, there is 
$x \in G$ such that $G-U(x)$ is contractible. The valuation
formula and $w(B(x))=1$ gives, using induction that
$1 = w(G-U(x))=w(G)+w(B(x))-w(S(x)) = w(G)+1-(1+(-1)^{(d-1)})$
so that $w(G) =1+(-1)^d$. 
\end{proof}

\paragraph{}
It follows: 

\begin{coro}
If all unit spheres $S(x)$ of $G$ have the same Euler characteristic,
then either $w(G)=0$ or $w(S(x))=0$ for all $x$. 
\end{coro}

\begin{proof} 
The reason is that if $w(S(x))=c$, then by the sphere formula, $\sum_x w(x) c = 0$.
This implies that either $c=0$ or then that $\sum_x w(x) = w(G)=0$. 
\end{proof} 

\paragraph{}
A complex $G$ is called {\bf homotopic to 1} if a sequence of contraction or inverse 
extension steps leads from $G$ to $1$. One of the simplest complexes which can 
not be contracted but which is homotopic to $1$ is the {\bf dunce hat}. There is an implementation
of that space with $f$-vector $(17, 52, 36)$. One could also look at more
general spheres, by replacing ``contractible" with ``homotopic to $1$". This would be 
{\bf impractical} however as checking whether a complex $G$ is homotopic to $1$ is difficult, 
even NP complete, while checking whether is contractible can be done in polynomial time. But still, 
also for this more general spaces, the Euler Gem formula holds. One can extend the
class even further by looking at Dehn-Sommerville spaces. 

\paragraph{}
A simplicial complex $G$ can be seen as a special {\bf CW complex} in which  simplices $x=x_k$ are added
along a time-line, starting with the zero-dimensional simplices, attached to
$(-1)$-spheres $0$, where $i(x)=1$, then adding $1$-dimensional simplices, seen
as attaching cells to $0$-spheres with  $i(x)=-1$, then adding $2$-dimensional cells etc.
As $i(x)=w(x)$ in each step, one can interpret $w(G)=\sum_x i(x)$ as a Poincar\'e-Hopf formula. 
If allowing cells to be attached $G_{k+1} = G_k+_{U_k} x_k$ to either contractible parts $S_k$ 
leading to $w(G_{k+1}) = w(G_k) + 1-w(S_k)=w(G-k)$ or by attaching new cells to spheres $S_k$,
in which case $w(G_{k+1})=w(G_k) + 1-w(S_k)=w(G_k) + (1+(-1)^{{{\rm dim}}(S(x_k))}$, 
the Morse function $f(x)=k$ if $x=x_k$ encodes the build-up and $i_f(x)=1-w(S^-_f(x))$
with $S^-_f(x)=\{ y \in S(x), f(y)<f(x)\}$ is the Poincar\'e-Hopf index. The dimension ${\rm dim}(B(x_k)$ is
the {\bf Morse index} and the Poincar\'e-Hopf formula $w(G)=\sum_x i_f(x)$ holds. 

\paragraph{}
For the {\bf cube CW complex} for example, one builds first the {\bf cube simplicial complex} with 8 vertices and 
12 edges, then adds 6 cells along circular complexes of Morse index $2$ and Poincar\'e-Hopf index $i_f(x)=1$.
The Poincar\'e-Hopf formula is now $w(G)=\sum_x i_f(x) = 8-12+6=2$, leading to the Euler formula $w(G)=v-e+f=2$ 
for Platonic solids, a discovery of Descartes. The Alexandroff topology defined by generalizing the notion of stars
$U(x)$ as a basis for open sets still attaches a finite topology to the CW complex. If a new index $x_k$ is added
then for every $x \in U(x_k)$, retroactively, $x_k$ is added to $U(x)$. The data structure of a CW complex now
is given by a set $G$, a basis for a finite topology on $G$ and a function $f:G \to \mathbb{N}$ which tells
how the structure is built-up. The data structure ``simplicial complex" is easier to work with but
is in general more costly as one has to deal with more cells. 
A CW cube only has 6 faces, 12 edges and 8 vertices, while a triangulated cube has 24 faces,
36 edges and 14=8+6 vertices

\section{Zero Characteristic}

\paragraph{}
The Euler-Gem formula and the sphere formula together immediately give the zero
Euler characteristic result: 

\begin{coro}[Corollary]
All odd-dimensional manifolds satisfy $w(G) =0$.
\end{coro}

\begin{proof}
By the Euler-Gem formula, for odd-dimensional
manifolds, every unit sphere has Euler characteristic $2$. The sphere formula
gives from this $0=\sum_x w(x) w(S(x)) = \sum_x w(x) 2 = w(G)$.
\end{proof}

\paragraph{}
As mentioned already, the proof shows that for any complex $G$ for
which the unit spheres $S(x)$ have constant Euler characteristic $c$, we either must have
that this constant $c=0$ or then that $G$ has zero Euler characteristic. We
for example can take a suspension of an arbitrary even-dimensional
manifold of Euler characteristic $2$. A small example in three dimensions is
a suspension of a copy of two projective planes. There are explicit implementations with
$f_G=(15, 42, 28)$. This is not a manifold but all
unit spheres have Euler characteristic $2$. The curvature is constant $0$.
An other 3 dimensional non-manifold case is the suspension of a disjoint union of
a torus and a sphere. The construction of more general classes is done below using
Dehn-Sommerville. 

\paragraph{}
The above corollary is usually proven using {\bf Poincar\'e-duality} which tells
that the {\bf Betti vector} $(b_0,b_1, \dots, b_{{\rm dim}(G)})$ 
is {\bf palindromic} $b_k=b_{{\rm dim}(G)-k}$. 
The {\bf Betti number} $b_k(G)$ is defined as the nullity of the block $L_k$ in the
$n \times n$ Hodge matrix $L = (d+d^*)^2$ of the complex $G$ with $n$ elements. The {\bf exterior derivative}
$d f(x) = \sum_{y \subset x,|x|-|y|=1} {\rm sign}(x|y) f(y)$ is a $n \times n$ matrix too satisfying $d^2=0$
so that $L=d d^* + d^* d$ is block diagonal. Functions restricted to $x$ with ${\rm dim}(x)=k$ are called 
{\bf $k$-forms}.  The matrices $d$ and so $L$ depend on an initial arbitrary orientation of the simplices 
which enters ${\rm sign}(x|y)$ which is defined to be $1$ if the orientation of $y$ matches the 
induced orientation of $x$ and $-1$ else. A complex is {\bf orientable} if one fix an orientation on all 
simplices by fixing the orientation on maximal simplices. 
But one does not need to have an orientable complex to define the derivative $d$. Changing the orientations is an 
orthogonal transformation on all forms and so is just a change of basis in the Hilbert space on which $d$ and $L$ acts. 
The spectrum and so the nullity $b_k(G)$ is not affected. 

\paragraph{}
The Poincar\'e duality only holds if $G$ is {\bf orientable}, meaning that the maximal simplices
can be oriented in a way such that the order is compatible on intersections.
For all simplicial complexes, one has the {\bf Euler-Poincar\'e} formula $\sum_k (-1)^k b_k=w(G)$.
This identity is just linear algebra using the rank-nullity theorem which
in this context is just the {\bf Hodge relation} relating the rank ${\rm ran}(D) = {\rm ran }(d) \oplus {\rm im}(d^*)$
of the {\bf Dirac operator} $D$ with the rank ${\rm ker}(D)$. If $d_k$ maps $k$ forms to $k+1$ forms and
$d_{k+1}^*$ maps $k+1$ forms to $k$ forms, one can see $b_k = {\rm dim}({\rm ker}(d_k/{\rm ran}(d_{k-1})$. If $f_k$
is the dimension of $k$-forms which is the number of $k$-simplices in $G$, this
immediately shows $\sum_k (-1)^k b_k = \sum_k (-1)^k f_k$. One can also derive the Euler-Poincar\'e formula by
looking ${\rm str}(e^{-tL})$ of the {\bf heat kernel} $e^{-t L}$, where ${\rm str}(A)=\sum_{x \in G} w(x) A(x,x)$ is the 
{\bf super trace}. McKean-Singer have pointed out that the non-zero eigenvalues on even and odd forms are the same
so that ${\rm str}(e^{-tL}) = w(G)$ is constant. But since $w(G)$ is an integer and for large $t$, we $e^{-t L}$ is
close to the kernel of $L$, the left hand side is for large $t$ equal to $\sum_k (-1)^k b_k$ and for $t=0$ we have
$\sum_k (-1)^k f_k$. 

\paragraph{}
If $H$ is an odd-dimensional manifold that is not orientable, one can look at the double cover $G$, 
apply Poincar\'e-duality to see $w(G)=0$, then use {\bf Riemann-Hurwitz} relation $w(H)=w(G/A) =w(G)/|A|$
for a group $A=\mathbb{Z}_2$ acting on $H$ without fixed points in $G$. The Riemann-Hurwitz formula is 
more general and also can take into account {\bf ramification points}, points for which the orbits of $A$
are smaller than in general. One can see many complexes as 
{\bf branched covers} $G$ of simpler complexes $H$ which can be seen as $H=G/A$ for a finite group $A$
acting on $G$. In any case, algebraic topology together with some algebraic geometry allows to see
the fact that odd-dimensional manifolds must have Euler characteristic zero. We will mention below an 
other simple classical approach using Morse theory and using that the indices of a Morse function  on
a manifold at a critical point satisfies $i_f(x)=-i_{-f}(x)$, if the dimension of the manifold is odd.

\paragraph{}
Why does the sphere theorem not apply for the cube or dodecahedron as
we have unit spheres at the vertices of constant Euler characteristic $3$? 
The reason is that one has to look at all the unit spheres in $G$ and not just at the 
unit spheres of $0$-dimensional parts of space. For the cube simplicial complex $G$,
we have eight $0$-dimensional points $x \in G$ and twelve $1$-dimensional points $x \in G$.
While $w(S(x))=3$ for the $0$-dimensional $x$ we have $w(S(x)=2$ for the 
$1$-dimensional $x \in G$. Now $8*3=24$ and $12*2=24$. The Euler characteristic is $8-12=-4$
which can be understood as the Euler characteristic of a $2$-sphere with $6$ holes so that $2-6=-4$. 
In the Dodecahedron case it is $2-12=-10$.

\paragraph{}
We will write in full generality the Euler characteristic in the next section as a sum of {\bf curvatures}
located on $0$-dimensional simplices. In the case of a $1$-dimensional complex, the curvature at a
vertex $x$ is $K(x) = 1-|S(x)|/2$. Summing up the curvatures $\sum_x K(x) = |V|-\sum_{x} |S(x)|/2
= |V|-|E|=f_0(G)-f_1(G)$ is just invoking the {\bf Euler-handshake formula} $\sum_v f_0(S(v)) = f_1(G)$
in any one dimensional complex $G$. In the case of the cube or Dodecahedron complex,
the curvature on the vertices is constant $-1/2$, leading to Euler characteristic $8*(-1/2)=-4$ 
or $20*(-1/2)=-10$ respectively.

\section{Gauss-Bonnet}

\paragraph{}
Let $f=(f_0, \dots, f_d)$ denote the {\bf $f$-vector} of $G$, where $f_k(G)$ is the 
number of elements in $G$ of length $k+1$. Define the {\bf simplex generating function}
$$   f_G(t) = 1+\sum_{k=0}^d f_k(G) t^{k+1} $$  
and simply call it the {\bf f-function}. Define 
$F_G(t)=\int_0^t f_G(s) \; ds$. While calculus is involved here, note that $f_G$ is 
a polynomial and that when dealing with polynomials, one stays in a finite setup by
just declaring $\frac{d}{dt} t^n=n t^{n-1}$ and $\int_0^t s^n ds = t^{n+1}/(n+1)$. Seen as
such, writing down derivatives and integrals is {\bf notation}. We do not invoke any limits
even so there calculus as a background theory interprets the expressions using limits. 
For the following, see \cite{dehnsommervillegaussbonnet}. Note that the curvature is located
on the 0-dimensional part of space.

\begin{thm}[Gauss-Bonnet]
$f_G(t)-1 = \sum_{v \in V} F_{S(v)}(t)$.
\end{thm}

\begin{proof}
Every $y \in S(x)$ carries a charge $t^{k+1}$, then $f_G(t)$ counts the total charge. 
Every $k$-simplex $y \in S(x)$ defines a $(k+1)$-simplex $z$ in $U(x) \subset G$
carrying the charge $t^{k+2}$. It contains $(k+2)$ vertices.
Distributing this charge equally to these points gives each a charge $t^{k+2}/(k+2)$.
The curvature $F_{S(x)}(t)$ adds up all the charges of $z$. 
\end{proof}


\paragraph{}
For $t=-1$, one gets a more traditional form as one can write Euler characteristic 
as a sum of curvatures 
$$   w(G) = \sum_{v \in V} K(v)  \; , $$
where $K(x)=F_{S(x)}(-1)$ is the {\bf Levitt curvature}, the discrete analogue of
the {\bf Gauss-Bonnet-Chern curvature} in the continuum \cite{Cycon}. 
An explicit formula for $K(v)$ with $f_{-1}=1$ is
$$ K(v) = \sum_{k=-1}^{d} (-1)^k \frac{f_k(S(v))}{k+2} $$
which appeared in \cite{Levitt1992} and be placed into the Gauss-Bonnet context in 
\cite{cherngaussbonnet} (We had not been aware of \cite{Levitt1992} when writing that paper). 
It is quite obvious once one realizes that
Euler characteristic is a total energy of the function $w(x)=(-1)^{{\rm dim}(x)}$ 
and that we can shove this value to the zero dimensional parts by placing $w(x)/({{\rm dim}(x)}+1)$
to each of the ${{\rm dim}(x)}+1$ vertices $v$ in $x$. When looking at the total value on 
a vertex $v$, we are interested in how much has been sent to us.

\paragraph{}
We have interpreted this curvature also as an expectation of Poincar\'e-Hopf indices
$$   K(v) = {\rm E}[ i_f(v)] $$ 
for example by taking the probability measure which assigns constant weights to all 
{\bf colorings} $f$. To formulate Poincar\'e-Hopf within simplicial complexes, one can start 
with a function $f:V \to R$ on vertices , where $R$ is an ordered ring like $\mathbb{Z}$ and 
$V=\{ x \in G, {\rm dim}(x)=0 \}$. Now define the index 
$i_f(v) = \sum_{x, v \in x {\rm is} \; {\rm max} \; on \; x} w(x)$. Because every simplex energy $w(x)$ 
has now must moved to the vertex in $x$ where $f$ was maximal, we have the Poincar\'e-Hopf
result 
$$  \sum_x w(x) = \sum_{v \in V} i_f(v) \; . $$
In the Gauss-Bonnet case, the value $w(x)$ has been distributed equally to each of the ${\rm dim}(x)+1$
vertices in $x$. We can now write $i_f(v) = 1-w(S^-_f(v))$, where $S^-_f(v)$ consists of all
simplices in $S(v)$ for which all $f$ values are smaller than $f(v)$. This set $S^-_f(v)$ is 
closed. 

\section{Dehn-Sommerville}

\paragraph{}
The {\bf $h$-function} $h_G(x) = (x-1)^d f_G(1/(x-1))$ generates coefficients $h_k$ of the form
$$ h_G(x) = h_0 + h_1 x + \cdots + h_d x^d + h_{d+1} x^{d+1} \; . $$
In other words, it is the generating function for the {\bf $h$-vector} 
$$  (h_0,h_1, \dots, h_{d+1}) \; . $$
$G$ is called {\bf Dehn-Sommerville} if this vector is {\bf palindromic}
meaning that $h_i=h_{d+1-i}$ for all $i=0, \dots, d+1$.

\begin{lemma}
$G$ is Dehn-Sommerville if and only if $f_G(t)$ satisfies
$f(t)+ (-1)^d f(-1-t)$, meaning that $g(t) = f(t-1/2)$ is 
either {\bf even} or {\bf odd}.
\end{lemma}

\begin{proof}
$h$ is palindromic if and only if the roots of
$h(t)=1+h_0 t + \cdots + h_d t^{d+1}  = (t-1)^d f(1/(t-1))$ are invariant under the
involution $t \to 1/t$. This is equivalent that the roots of $f$ are invariant
under the involution $t \to -1-t$ and so to the symmetry $f(-1-t)= \pm f(t)$
for the $f$-function. If $G$ is a complex with maximal dimension $d$ and $f_G$
satisfies $f(t)+ (-1)^d f(-1-t)$ then $f(-1) = (-1)^d f(0)$ so that 
$w(G)=1+f(-1) =1+(-1)^d$.
\end{proof}

\paragraph{}
Let $G+H = G \cup H \cup \{ x+y, x \in G, y \in H\}$ be the {\bf join} of $G$ and $H$. 
Since $f_{G+H}(t) = f_G(t) f_H(t)$, we immediately see that
if $G$ and $H$ are Dehn-Sommerville, then the join $G+H$ is Dehn-Sommerville again.
Also Barycentric refinements and connected sums of 
$G$ and $H$ along a sphere $S$ are Dehn-Sommerville.
Also edge refinements of Dehn-Sommerville spaces are 
Dehn-Sommerville. 

\begin{coro}
Odd-dimensional Dehn-Sommerville spaces and especially, odd-dimensional 
manifolds all have zero Euler characteristic. 
\end{coro}

\paragraph{}
Define $\mathcal{X}_{(-1)}=\{ \}$ and inductively
$\mathcal{X}_d = \{ G \; | \; w(G) = 1+(-1)^d,  
  S(x) \in \mathcal{X}_{d-1}, \forall x \in G \}$. 
If $G,H \in \mathcal{X}$, then $G+H \in \mathcal{X}$
because $S_{G+H}(x) = S(x) + H$ for $x \in G$ and 
$S_{G+H}(x) = G + S(x)$  for $x \in H$. Also the assumption 
on Euler characteristic works as 
$w(G) = 1-f_G(-1) \in \mathbb{Z}_2 = \{ -1,1 \}$
and $w(G+H) = f_G(-1) f_G(-1)$ is still in $\{ -1,1 \}$.

\begin{coro}
Every $d$-manifold of Euler characteristic $1+(-1)^d$ is 
Dehn-Sommerville.  
\end{coro}

\begin{proof}
The reason is that the unit spheres are spheres and so Dehn-Sommerville. The Gauss-Bonnet
formula shows that the $f$-vector of $G$ as an integral of Dehn-Sommerville
expressions satisfies the Dehn-Sommerville relation. 
\end{proof} 

\paragraph{}
Dehn-Sommerville is  remarkable. For example, for any 4-sphere,
the $f$-vector satisfies
$$ -22 f_1+33 f_2-40 f_3+45 f_4=0 \; . $$
For example, for the smallest $4$-sphere
with $f=(f_0,f_1,f_2,f_3,f_4)=(10, 40, 80, 80, 32)$
one has
$$ (10, 40, 80, 80, 32) \cdot (0,-22,33,-40,45)=0  \; . $$

\section{Poincar\'e-Hopf}

\paragraph{}
Again, let $V=\subset G$ be the set of $0$-dimensional sets in $G$. 
It can naturally be identified with $\bigcup_x x$ (even so it is not the same.
The object $\{ \{v\} \}$ is not the same than $\{ v\}$). A function on $V$ is 
{\bf locally injective} $f(v) \neq f(w)$ for every $w \in S(x) \cap V$. 
For such a function, define $S^-_f(v) = \{ x \in S(v), f(v)<f(w)\; \forall w \in V \cap S(v) \}$ 
and $S^+_f(v) = \{ S(v), f(v)>f(w) \; \forall w \in V \cap S(v) \} = S^-_{-f}(v)$. 
The Poincar\'e-Hopf theorem assures that
$$  w(G) = \sum_{w \in V \cap S(v)} i_f(w) = \sum_{w \in V \cap S(v)} i_{-f}(w) \; . $$

\paragraph{}
Following \cite{indexformula}, we have now
$$  w(G) = \sum_{w \in V \cap S(v)} j_f(w)  \; , $$
where $j_f(w) = (i_f(w)  + i_{-f}(w))/2$ is the average of the two indices.
The valuation formula  gives
$$  w(S^-(v)) + w(S^+(v)) = w(S(v)) - w(C(v)) \; , $$
where $C(v)$ are the simplices $x$ in $S(v)$ on which $f-f(v)$ changes sign. This
naturally becomes a sub simplicial complex of the Barycentric refinement of $G$, 
by looking at the elements as vertices in a graph and taking
the Whitney complex. The valuation formula is equivalent to 
$$ w(S(v))-2-j_f(v) = w(C_f(v))  \; . $$
We called $C_f(v)$ the {\bf center manifold} of $f$ at $v$. In classical Morse theory, 
this is the manifold $\{w \in S_r(v), f(w)=f(v) \}$ in a small sphere $S_r(v)$ around a 
critical point of $f$ which typically is a $(d-2)$-manifold or empty by the 
{\bf classical Sard theorem}.

\paragraph{}
For Morse functions $f$, the space $C_f(v)$ is actually a $(d-2)$-manifold or empty. 
For a Morse function on a manifold, one also has the 
{\bf symmetric index} $j_f(v) = i_f(v)+i_{-f}(v)=0$ at every 
critical point and so can deduce the zero Euler characteristic result again. 
The Barycentric refinement $G_1$ is the Whitney complex of the graph in which 
two sets in $G$ are connected if one is contained in the other:

\begin{thm}[Discrete Sard]
If $G$ is a $d$-manifold and $f:G \to R$ is locally injective, then 
for any $c \neq f(V)$, the $\{ f =c \}$ is a discrete $(d-1)$-manifold in $G_1$. 
\end{thm}
\begin{proof}
This can be shown by induction with respect to dimension. In order to show that 
the unit sphere $S(x)$ is a $(d-1)$ sphere, note that by induction the 
intersection $\{f = c\}$ in every unit sphere (a (d-1) sphere) is a $d-2$
dimensional manifold. But we have more: in general in $G_1$, every unit sphere is
the join of $S^-(x) = \{ y \in S(x), y \subset x \}$ and $S^+(x) = \{ y \in S(x), x \subset y \}$. 
Now, since $f-f(x)$ changes sign on $S^+(x)$ but not on $S^-$
by induction, the level surface in $S^+(x)$ is a sphere of dimension $1$ lower. 
The sphere $S(x) \cap \{ f = c\}$ in the level surface of $G$ is now the 
join of $S^-(x)$ and the level $\{ f = c\}$ in $S^+(x)$ which is a sphere of co-dimension
$1$ in $S^+(x)$. 
\end{proof} 

\paragraph{}
The symmetric index of a locally injective function $f$ is 
$j_f(v) = [2-w(S(v)) - w(C_f(v))]/2$, where $C_f(v)$ is the center manifold
in $S(v)$. (See \cite{indexformula}) 

\paragraph{}
If $G$ is an odd-dimensional d-manifold, then $S(v)$ is an even dimensional sphere
and $C_f (v)$ is a $(d-2)$-dimensional manifold by the {\bf discrete Sard theorem}.
We haven then $w(S(v))-2=0$ because $S(v)$ was an even dimensional sphere.
The symmetric index is 
$$ j_f(v) = w(C(v)) = 0  \; . $$
We see that index expectation by averaging two functions $f,-f$ gives us a 
curvature that is constant $0$. The conclusion is again that $w(G)=0$. 

\paragraph{}
The index formula is also interesting in even dimensions: for a 4-manifold
for example, one can see the Gauss-Bonnet curvature as the expectation of 
$j_f(v) = 1-w(C_f(v))/2$. This allows to see the Euler characteristic of a $4$-manifold
in terms of the expectation of Euler characteristic of 
``random two dimensional center manifolds" $C_f(x)$. 
In the positive curvature case, $C_v(v)$ is connected which from the classification of
manifolds shows that $w(C_f(v)) \leq 2$ so that $j_f(v) \geq 0$, corresponding to 
the observation of Milnor \cite{Chern1966} that the Gauss-Bonnet Chern curvature is 
non-negative in the positive curvature case for $4$-manifolds. 
This argument is no more available in dimensions 6 and higher as the 
curvature can have become negative \cite{Geroch,Klembeck}. 
See \cite{Hopf1932,BergerPanorama,BishopGoldberg}.
The index analysis suggests to look for probability spaces of functions for which 
$j_f(v) \geq 0$ and not being constant $0$. This is equivalent to $w(C_f(v)) \leq 2$.

\section{Some references}

\paragraph{}
The notion of finite abstract simplicial complex is due to 
Dehn and Heegaard \cite{DehnHeegaard,BurdeZieschang,MunkholmMunkholm}. 
The Dehn-Sommerville relations go back to Dehn and Sommerville 
\cite{Sommerville1927,Klee1964,NovikSwartz,MuraiNovik,LuzonMoron,BrentiWelker,Hetyei,Klain2002,BergerLadder}.
Euler characteristic was considered first for Platonic solids and
experimentally first studied by Descartes. Euler gave the gem
formula in the case $d=2$ and for graphs which are planar. \cite{Klee63,lakatos,Gruenbaum2003,Richeson}. 
Euler characteristic has been seen in the context of
invariant valuations \cite{Hadwiger,KlainRota}. It is the only valuation in the 
$({\rm dim}(G)+1)$-dimensional space of valuations which is normalized and invariant under Barycentric refinements.
For finite topological spaces related to simplicial complexes \cite{Alexandroff1937,Stong1965,May2008}.
For McKean-Singer in the continuum \cite{McKeanSinger,Cycon}, for the classical
Gauss-Bonnet theorem in higher dimension \cite{Allendoerfer,AllendoerferWeil,Fenchel,Chern44,Rosenberg,Cycon} and 
in the discrete \cite{Levitt1992}. For Poincar\'e-Hopf, \cite{poincare85,hopf26,Morse29}.
For discrete notions of Morse theory \cite{forman95,forman98,forman2000,Forman2002,Forman1999}
within {\bf discrete Morse theory}. For notions of spheres \cite{I94,I94a,Evako1994}
within {\bf digital topology} \cite{Evako2013} or \cite{BobenkoSuris} for discrete differential geometry.
While topologists call geometric realizations of simplicial complexes ``Polytopes", most of the 
polytop literature considers {\bf convex polytopes} and so geometric realizations of $d$-spheres
\cite{Schlafli,coxeter,gruenbaum,symmetries,Ziegler}.
Discrete notions of homotopy was considered already in the discrete \cite{Whitehead} and Evako. 
A crucial simiplification occured in\cite{CYY}. Probabilistic notions in geometry like 
{\bf integral geometry} go back to \cite{Blaschke,Santalo,Banchoff1967,Banchoff1970,Nicoalescu}.
The join in graph theory was introduced in \cite{Zykov}.
For the history of manifolds, see \cite{HistoryTopology,Scholz}. In the context of Dehn-Sommerville, the 
{\bf arithmetic of graphs} comes in. While manifolds are preserved by disjoint union, they are not preserved
by the join operation. But spheres, and more generally Dehn-Sommerville spaces are join monoids.
For graph multiplication, see \cite{Shannon1956,Sabidussi,ImrichKlavzar,HammackImrichKlavzar}.
The Hopf conjecture has very early on seen in the context of Gauss-Bonnet
\cite{HopfCurvaturaIntegra,Allendoerfer,Fenchel,AllendoerferWeil,Chern44,Chern1990,Cycon}.
The Hopf conjecture \cite{Hopf1932} have reappeared in the sixties \cite{BishopGoldberg} and \cite{Chern1966}
and are listed as problems 8) and 10) in the problem collection \cite{YauSeminar1982}.
For historical remarks on Gauss-Bonnet \cite{Chern1990}, on manifolds 
\cite{Dieudonne1989,HistoryTopology}. 
Discrete notions in curvature have appeared first in \cite{Eberhard1891}.
It was then used in graph coloring contexts like \cite{Heesch}. Variants have appeared in two dimensions
\cite{Gromov87,Presnov1990,Presnov1991,Higuchi, NarayanSaniee,RetiBitayKosztolanyi}. 

\paragraph{}
We have explored the topic in the last couple of years, often in a graph theoretical
frame work which is almost equivalent as every graph comes naturally with a {\bf Whitney simplicial complex}
coming from the vertex sets of complete subgraphs and every simplicial complex $G$ defines a graph in which $G$
are the vertices and two are connected if one is contained in the other.
Many texts treat graphs as one-dimensional simplicial complexes, meaning that the simplicial complex
on the graph is the 1-skeleton complex $V \cup E$. In topological graph theory, one 
studies graphs embedded in surfaces \cite{TuckerGross} and so deals with 2-dimensional CW complexes, where
the connected components in an embedding serve as 2-cells. There are strong links between graphs and
simplicial complexes because there are various ways to get also higher dimensional simplicial complexes 
from a graph. The Whitney is the most natural one as all the connectivity, geometric, differential geometric,
or cohomological properties match what one expects in geometric realizations. 
\cite{FerrarioPiccinini,MunkresAlgebraicTopology} and especially \cite{JonssonSimplicial}. 

\paragraph{}
For the Poincare-Hopf which is related to the energy theorem, see
\cite{poincarehopf,parametrizedpoincarehopf,PoincareHopfVectorFields,MorePoincareHopf}. 
For the energy theme, see 
\cite{CountingMatrix,GreenFunctionsEnergized,EnergizedSimplicialComplexes,KnillEnergy2020}. 
For index expectation, see
\cite{indexexpectation,colorcurvature,indexformula,DiscreteHopf2,ConstantExpectationCurvature}. 
For the Sard theorem, see \cite{KnillSard} which tried to be close to the classical Sard theorem \index{Sard42}.
For the Gauss-Bonnet theorem, see 
\cite{elemente11,cherngaussbonnet,valuation,DehnSommerville,dehnsommervillegaussbonnet}. 
For some discrete work on Hopf type questions \cite{DiscreteHopf,DiscreteHopf2}
and constant curvature \cite{ConstantExpectationCurvature}.
For the Hodge or McKean-Singer theme see
\cite{knillmckeansinger,DiracKnill}
For the theme of finite topologies on graphs or complexes
\cite{KnillTopology,KnillTopology2023}.
For overview attempts, see 
\cite{knillcalculus,AmazingWorld,KnillBaltimore}.
For our own explorations on the arithmetic of graphs including the Zykov-Sabidussi ring and 
its dual, the Shannon ring \cite{ArithmeticGraphs,RemarksArithmeticGraphs,ComplexesGraphsProductsShannonCapacity}.

\section{Code}

\paragraph{}
Here is some code which allows an interested reader to experiment with some
of the notions which appeared. The computer uses the 
Whitney functor to nicely generates random complexes from random graphs. With
the parameters given, the random complexes produced are typically 4-5 dimensional.

\begin{tiny}
\lstset{language=Mathematica} \lstset{frameround=fttt}
\begin{lstlisting}[frame=single]
Cl[A_]:=If[A=={},{},Delete[Union[Sort[Flatten[Map[Subsets,A],1]]],1]];  
Fvector[G_]:=If[Length[G]==0,{},Delete[BinCounts[Map[Length,G]],1]]; 
Ffunction[G_,t_]:=Module[{f=Fvector[G],n},Clear[t]; n=Length[f];
  If[Length[G]==0,1,1+Sum[f[[k]]*t^k,{k,n}]]];
Whitney[s_]:=If[Length[EdgeList[s]]==0,Map[{#}&,VertexList[s]],
  Map[Sort,Sort[Cl[FindClique[s,Infinity,All]]]]];
U[G_,x_]:=Module[{u={}},
  Do[If[SubsetQ[G[[k]],x],u=Append[u,G[[k]]]],{k,Length[G]}];u];
Stars[G_]:=Table[U[G,G[[k]]],{k,Length[G]}];
Spheres[G_]:=Table[u=U[G,G[[k]]];Complement[Cl[u],u],{k,Length[G]}]
w[x_]:=-(-1)^Length[x]; Chi[A_]:=Total[Map[w,A]]; 
g[G_]:=Module[{V=Stars[G],n=Length[G]},Table[w[G[[k]]]*w[G[[l]]]*
  Chi[Intersection[V[[k]],V[[l]]]],{k,n},{l,n}]];
sg[G_]:=Module[{V=Spheres[G],n=Length[G]},Table[w[G[[k]]]*w[G[[l]]]*
  Chi[Intersection[V[[k]],V[[l]]]],{k,n},{l,n}]];
Curvature[G_,t_]:=Module[{h=Ffunction[G,y]},Integrate[h,{y,0,t}]];
Curvatures[G_,t_]:=Module[{S=Spheres[G]},Table[If[Length[G[[k]]]==1,
  Curvature[S[[k]],t],0],{k,Length[S]}]];
Levitt[G_] := -Curvatures[G, t] /. t -> (-1);
Zykov[A_,B_]:=Module[{q=Max[Flatten[A]],Q,G=A},
  Q=Table[B[[k]]+q,{k,Length[B]}];
  Do[G=Append[G,Union[A[[a]],Q[[b]]]],{a,Length[A]},{b,Length[Q]}];
  G=Union[G,Q]; If[A=={},G=B]; If[B=={},G=A]; G];
DehnSommervilleQ[G_]:=Module[{f},Clear[t];f=Ffunction[G,t];
   Simplify[f] === Simplify[(f /. t->-1-t)]];

s=RandomGraph[{20,100}]; G=Whitney[s]; K=Levitt[G]; Q=sg[G]; 
Print["f-vector  ",Fvector[G]]; 
Print["Gauss-Bonnet ", Total[K]==Chi[G]];
H=Whitney[CycleGraph[4]]; sphere3=Zykov[H,H]; sphere5=Zykov[sphere3,H];
Print["3- sphere is flat ",Union[Levitt[sphere3]]=={0}];
Print["5- sphere is flat  ",Union[Levitt[sphere5]]=={0}];
Print["Green-Star matrix is unimodular. Det=  ",Det[g[G]]]; 
Print["nullity of sphere Green matrix:  ",Length[NullSpace[Q]]]; 
Print["sphere formula  ",Sum[w[G[[k]]]*Q[[k,k]],{k,Length[G]}]==0];
Print["3-sphere is Dehn-Sommerville ",DehnSommervilleQ[sphere3]];
Print["5-sphere is Dehn-Sommerville ",DehnSommervilleQ[sphere5]];
Print["disk is not D-S ",Not[DehnSommervilleQ[Whitney[StarGraph[5]]]]];
\end{lstlisting}
\end{tiny}

One open question is the significance of the relative large kernel of the
{\bf Green sphere matrix} $s(x,y) = w(x) w(y) w(S(x) \cap S(y))$ which is a bit
surprising, given that $g(x,y) = w(x) w(y) w(U(x) \cap U(y))$ is always unimodular.
In experiments we tried to correlate the nullity with the f-vector. 

\bibliographystyle{plain}

\end{document}